\documentclass[11pt]{article}

\usepackage{amssymb,amsmath}
\usepackage{amsthm}
\usepackage{hyperref}
\usepackage{mathrsfs}
\usepackage{textcomp}
\hypersetup{
    colorlinks,%
    citecolor=black,%
    filecolor=black,%
    linkcolor=black,%
    urlcolor=black
}

\usepackage{verbatim}

\usepackage{sectsty}

\makeatletter

\newdimen\bibspace
\renewenvironment{thebibliography}[1]{%
 \section*{\refname 
       \@mkboth{\MakeUppercase\refname}{\MakeUppercase\refname}}%
     \list{\@biblabel{\@arabic\c@enumiv}}%
          {\settowidth\labelwidth{\@biblabel{#1}}%
           \leftmargin\labelwidth
           \advance\leftmargin\labelsep
           \itemsep\bibspace
           \parsep\z@skip     %
           \@openbib@code
           \usecounter{enumiv}%
           \let\p@enumiv\@empty
           \renewcommand\theenumiv{\@arabic\c@enumiv}}%
     \sloppy\clubpenalty4000\widowpenalty4000%
     \sfcode`\.\@m}
    {\def\@noitemerr
      {\@latex@warning{Empty `thebibliography' environment}}%
     \endlist}

\makeatother

\makeatletter

\newtheorem{thm}{Theorem}[section]
\newtheorem{lem}[thm]{Lemma}
\newtheorem{prop}[thm]{Proposition}

\def\Xint#1{\mathchoice
  {\XXint\displaystyle\textstyle{#1}}%
  {\XXint\textstyle\scriptstyle{#1}}%
  {\XXint\scriptstyle\scriptscriptstyle{#1}}%
  {\XXint\scriptscriptstyle\scriptscriptstyle{#1}}%
  \!\int}
\def\XXint#1#2#3{{\setbox0=\hbox{$#1{#2#3}{\int}$}
  \vcenter{\hbox{$#2#3$}}\kern-.5\wd0}}

\def\dashint{\Xint-}

\newcommand{\al}{\alpha}                \newcommand{\lda}{\lambda}
\newcommand{\om}{\Omega}                \newcommand{\pa}{\partial}
\newcommand{\va}{\varepsilon}           \newcommand{\ud}{\mathrm{d}}
\newcommand{\be}{\begin{equation}}      \newcommand{\ee}{\end{equation}}
                 \newcommand{\X}{\overline{X}}
              \newcommand{\B}{\mathcal{B}}
\newcommand{\R}{\mathbb{R}}

\newcommand{\dlim}{\displaystyle\lim}

\begin{document}

\title{\textbf{Local analysis of solutions of fractional semi-linear elliptic equations with isolated singularities}
\bigskip}

\author{\medskip  Luis Caffarelli, \ \ Tianling Jin, \  \  Yannick Sire, \ \
Jingang Xiong}

\date{\today}

\maketitle

\begin{abstract} In this paper, we study the local behaviors of nonnegative local solutions of fractional order semi-linear equations $(-\Delta )^\sigma u=u^{\frac{n+2\sigma}{n-2\sigma}}$ with an isolated singularity, where $\sigma\in (0,1)$. We prove that all the solutions are asymptotically radially symmetric. When $\sigma =1$, these have been proved in \cite{CGS} by Caffarelli, Gidas and Spruck .

\end{abstract}

\section{Introduction}

In this paper, we study the local behaviors of nonnegative solutions of 
\be\label{eq:maineq}
(-\Delta)^\sigma u=u^{\frac{n+2\sigma}{n-2\sigma}} \quad \mbox{in } B_1 \setminus \{0\}
\ee
with an isolated singularity at the origin, where the punctured unit ball $B_1\setminus \{0\}\subset\R^n$ with $n\ge 2$, $\sigma\in (0,1)$, and $(-\Delta)^\sigma $ is the fractional Laplacian. This semi-linear equation involving the fractional Laplacian with the critical Sobolev exponent arises in contexts such as the Euler-Lagrangian equations of Sobolev inequalities \cite{Lie83, CLO, Li04}, a fractional Yamabe problem \cite{GQ, GMS, JX1}, a fractional Nirenberg problem \cite{JLX, JLX2} and so on. A feature of \eqref{eq:maineq} is that it is conformally invariant, and one may refer to \cite{GZ, CG} for its connections to conformal geometry. 

Singular solutions of fractional order conformal Laplacian equations was studied in \cite{GMS}, where the authors investigated the singular sets of such solutions and characterized the connection between the dimension of the singular sets and the order of the equations. Solutions of \eqref{eq:maineq} with an isolated singularity are the simplest examples of those singular solutions. We are interested in the local behaviors of solutions of \eqref{eq:maineq} near the singularity such as their precise blow up rates and asymptotically radial symmetry property. In the classical case $\sigma=1$, this was proved in the pioneer paper \cite{CGS} of Caffarelli, Gidas and Spruck.

We analyze \eqref{eq:maineq} via the extension formulations for fractional Laplacians established by Caffarelli and Silvestre \cite{CaS}. This is a commonly used tool nowadays, through which instead of \eqref{eq:maineq} we can study a degenerate elliptic equation with a Neumann boundary condition in one dimension higher. We use capital letters, such as $X=(x,t)\in\R^n\times\R_+$, to denote points in $\R^{n+1}$. We also denote $\B_R$ as the ball in $\R^{n+1}$ with radius $R$ and center at the origin, $\B^+_R$ as the upper half ball $\B_R\cap \R^{n+1}_+$, and $\pa' \B_R$ as the flat part of $\pa \B_R$ which is the ball $B_R$ in $\R^{n}$. Then the substitution equation we study is
\be\label{eq:ex0}
\begin{cases}
\begin{aligned}
\mathrm{div}(t^{1-2\sigma} \nabla U)&=0 & \quad &\mbox{in }\B_2^+,\\
\frac{\pa U}{\pa \nu^\sigma} (x,0)&= U^{\frac{n+2\sigma}{n-2\sigma}}(x,0) &\quad& \mbox{on }\pa' \B_2\setminus\{0\},
\end{aligned}
\end{cases}
\ee
where $\frac{\pa U}{\pa \nu^\sigma}(x,0)= -\lim_{t\to 0^+} t^{1-2\sigma} \pa_t U(x,t)$. By the extension formulation in \cite{CaS}, we only need to analyze the behaviors of the traces $$u(x):=U(x,0)$$ of the nonnegative solutions $U(x,t)$ of \eqref{eq:ex0} near the origin, from which the behaviors of solutions of \eqref{eq:maineq} follow. 

We say that $U$ is a nonnegative solution of \eqref{eq:ex0} if $U$ is in the weighted Sobolev space $W^{1,2}(t^{1-2\sigma}, \B_2^+\setminus \overline \B_\va^+)$ for all $\va>0$, $U\ge 0$, and it satisfies \eqref{eq:ex0} in the sense of distribution away from $0$ (see \cite{JLX} for more details on this definition). Then it follows from the regularity result  in \cite{JLX} that $U(x,t)$ is locally H\"older continuous in $\overline \B_1\setminus\{0\}$. We say that the origin $0$ is a non-removable singularity of solution $U$ of \eqref{eq:ex0} if $U(x,0)$ can not be extended as a continuous function near the origin. 
 Our first result shows its precise blow up rate near non-removable isolated singularities.
\begin{thm}\label{thm:a} 
Suppose that $U$ is a nonnegative solution of \eqref{eq:ex0}. Then either $u$ can be extended as a continuous function near $0$, or there exist two positive constants $c_1$ and $c_2$ such that
\be\label{eq:low and up bound}
c_1|x|^{-\frac{n-2\sigma}{2}}\le u(x)\le c_2|x|^{-\frac{n-2\sigma}{2}}.
\ee
\end{thm}
We remark that once \eqref{eq:low and up bound} holds, the Harnack inequality \eqref{eq:spherical harnack} implies that
\[
C_1|X|^{-\frac{n-2\sigma}{2}}\le U(X)\le C_2|X|^{-\frac{n-2\sigma}{2}}
\]
holds as well, for some positive constants $C_1$ and $C_2$.

We are also able to show that the trace $u$ of every solution $U$ of \eqref{eq:ex0} is asymptotically radially symmetric.
\begin{thm}\label{thm:b} If $U$ is a nonnegative solution of \eqref{eq:ex0}, then
\[
u(x)=\bar u(|x|)(1+O(|x|))\quad\mbox{as }x\to 0,
\]
where $\bar u(|x|)=\dashint_{\mathbb{S}^n}u(|x|\theta)\ud \theta$ is the spherical average of $u$. 
\end{thm}

When $\sigma =1$, Theorem \ref{thm:a} and Theorem \ref{thm:b} were proved in \cite{CGS} by Caffarelli, Gidas and Spruck. We may also see \cite{KMPS} for this classical case, and \cite{Li06, HLT} for some conformally invariant fully nonlinear equations with isolated singularities. A similar upper bound in \eqref{eq:low and up bound} was obtained in \cite{GMS} under additional assumptions that the equations are globally satisfied on the \emph{whole} space and the conformal metric is complete. Also, 
the upper bound in \eqref{eq:low and up bound} should hold for solutions of equation \eqref{eq:ex0} with other singular sets which are small in some capacity sense instead of one single point, see, e.g., \cite{ChLin} for the case of $\sigma=1$. 

In the global case that the origin is a non-removable isolated singularity, then the solutions are cylindrically symmetric with respect to the origin.
\begin{thm}\label{thm:gs}
Let $U$ be a nonnegative solution of
\be\label{eq:ex01}
\begin{cases}
\begin{aligned}
\mathrm{div}(t^{1-2\sigma} \nabla U)&=0 & \quad &\mbox{in }\R^{n+1}_+,\\
\frac{\pa U}{\pa \nu^\sigma} (x,0)&= U^{\frac{n+2\sigma}{n-2\sigma}}(x,0) &\quad &\mbox{on }\R^n\setminus\{0\}.
\end{aligned}
\end{cases}
\ee
Suppose the origin $0$ is not removable. Then $U(x,t)=U(|x|, t)$ and $\pa_r U(r,t)< 0$ for all $0<r<\infty$.
\end{thm}

An example solution of \eqref{eq:ex01} is $u(x)=|x|^{-\frac{n-2\sigma}{2}}$ and $U(x,t)$ is the Poisson integral of $u(x)$. If solutions of \eqref{eq:ex01} can be extended as continuous functions near the origin, then one can show the second equation in \eqref{eq:ex01} holds in $\R^n$ and all the solutions have been classified in Theorem 1.5 of \cite{JLX}.

Similar arguments can also be applied to obtain a Harnack type inequality, which is due to Schoen \cite{Schoen} when $\sigma =1$.
\begin{thm}\label{thm:c} 
Suppose that $U$ is a nonnegative solution of
\be\label{eq:nonsingu}
\begin{cases}
\begin{aligned}
\mathrm{div}(t^{1-2\sigma} \nabla U)&=0 & \quad& \mbox{in }\B_{3R}^+,\\
\frac{\pa U}{\pa \nu^\sigma} (x,0)&= U^{\frac{n+2\sigma}{n-2\sigma}}(x,0) &\quad& \mbox{on }\pa' \B_{3R},
\end{aligned}
\end{cases}
\ee
for some $R>0$. Then
\be \label{eq:harnack}
\max_{B_R} u\min_{B_{2R}} u\leq CR^{2\sigma-n},
\ee
where $C$ depends only on $n$ and $\sigma$. 
\end{thm}

Since \eqref{eq:maineq} is conformally invariant, \eqref{eq:ex0} is also invariant under those Kelvin transformations with respect to the balls centered on $\pa\R^{n+1}_+$. More precisely,
for each $\bar x\in\R^n$ and $\lda>0$, we define, $\X=(\bar x,0)$, and
\be\label{kelvin}
U_{\X, \lda}(\xi):=\left(\frac{\lda}{|\xi-\X|}\right)^{n-2\sigma}U\left(\X+\frac{\lda^2(\xi-\X)}{|\xi-\X|^2}\right),
\ee
the Kelvin transformation of $U$ with respect to the ball $\B_{\lda}(\X)$. If $U$ is a solution of \eqref{eq:ex0}, then $U_{\bar X,\lda}$ is a solution of \eqref{eq:ex0} in the corresponding domain. This allows us to use the moving sphere method introduced by Li and Zhu \cite{LZhu}. And we adapt some arguments from \cite{Li06}. But there are extra difficulties.  One is the degeneracy of \eqref{eq:ex0}. The others would be those extra efforts to obtain the estimates of $U$ from those of its trace $u$. Furthermore, the estimates for $u$ inherited from $U$ sometimes are too weak to apply.

A further goal would be to show that the trace of every solution of \eqref{eq:ex0} with a non-removable singularity the origin is asymptotically close to the trace of a global solution of \eqref{eq:ex01}. This is true when $\sigma=1$ and it was proved in \cite{CGS} and \cite{KMPS}. And it is also true in the case of some conformally invariant fully nonlinear equations \cite{HLT}. A missing ingredient in our case to analyze the solutions of \eqref{eq:ex01} is the ODEs analysis compared to the case when $\sigma=1$. We plan to continue in future work.

The paper is organized as follows. Section \ref{sec: preliminaries} includes two propositions: a Harnack inequality and a maximum principle, which will be used often throughout the paper. In Section \ref{sec:bd}, we obtain the blow up upper and lower bounds and prove Theorem \ref{thm:a}. Theorem \ref{thm:gs} on cylindrical symmetry of global solutions of \eqref{eq:ex01} is proved in Section \ref{sec:global}. Section \ref{sec:asymptotical} is devoted to asymptotically radial symmetry property of solutions of \eqref{eq:ex0}. Finally, we prove the Harnack inequality \eqref{eq:harnack} in Section \ref{sec:harnack}.

\bigskip

\noindent\textbf{Acknowledgements:} L. Caffarelli was supported in part by NSF grants DMS-1065926 and DMS-1160802. Y. Sire was supported in part by grant ANR ``HAB" and ERC grant ``EPSILON". J. Xiong was supported in part by the First Class Postdoctoral
Science Foundation of China (No. 2012M520002). Both T. Jin and J. Xiong thank Professor Y.Y. Li for helpful discussions and constant encouragement, and for informing us a useful lemma in \cite{LL} .

\section{Preliminaries} \label{sec: preliminaries}

We begin with introducing some notations and some propositions which will be used in our arguments. We denote $\B_R(X)$ as the ball in $\R^{n+1}$ with radius $R$ and center $X$, $\B^+_R(X)$ as $\B_R(X)\cap \R^{n+1}_+$, and $B_R(x)$ as the ball in $\R^{n}$ with radius $R$ and center $x$. We also write $\B_R(0), \B^+_R(0), B_R(0)$ as $\B_R, \B_R^+, B_R$ for short respectively. For a domain $D\subset \R^{n+1}_+$ with boundary $\pa D$, we denote $\pa' D$ as the interior of $\overline D\cap \pa \R^{n+1}_+$ in $\R^n=\partial\R^{n+1}_+$ and $\pa''D=\pa D \setminus \pa' D$. Thus, $\pa' \B_R^+(X)=\B_R(X)\cap\R^n=B_R$, and $\pa'' \B_{R}^+(X)=\pa \B_{R}(X)\cap\R^{n+1}_+$.

We say
$U\in W^{1,2}_{loc}(t^{1-2\sigma},\overline {\R^{n+1}_+})$ if $U\in W^{1,2}(t^{1-2\sigma},\B_R^+)$ for every $R>0$, and $U\in W^{1,2}_{loc}(t^{1-2\sigma},\overline {\R^{n+1}_+}\setminus\{0\})$ if $U\in W^{1,2}(t^{1-2\sigma},\B_R^+\setminus\overline \B_\va^+)$ for any all $R>\va>0$. 

The following two propositions will be used frequently in our arguments, whose proofs can be found in \cite{JLX}. We state them here for convenience. The first is a Harnack inequality (see also \cite{CS, TX}).

\begin{prop}\label{prop:harnack} Let $U \in W^{1,2}(t^{1-2\sigma}, \B_1^+)$ be a nonnegative weak solution of
\[
\begin{cases}
\begin{aligned}
\mathrm{div}(t^{1-2\sigma} \nabla U)&=0 & \quad &\mbox{in }\B_1^+,\\
\frac{\pa U}{\pa \nu^\sigma}(x,0) &= a(x) U(x,0) &\quad &\mbox{on } B_1.
\end{aligned}
\end{cases}
\]
If $a\in L^p(B_1)$ for some $p>\frac{n}{2\sigma}$, then we have
\[
\sup_{ \B_{1/2}^+} U\leq C \inf_{ \B_{1/2}^+} U,
\]
where $C$ depends only on $n,\sigma$ and $\|a\|_{L^p(B_{1})}$.
\end{prop}

The second one is on a maximum principle for positive supersolutions with an isolated singularity.
\begin{prop}[Proposition 3.1 in \cite{JLX}]\label{prop:liminf}  
Suppose that $U\in W^{1,2}(\B_1^+\setminus\overline \B_\va)$ is a solution of
\[
\begin{cases}
\begin{aligned}
\mathrm{div}(t^{1-2\sigma} \nabla U)&\le 0 & \quad& \mbox{in }\B_1^+,\\
\frac{\pa U}{\pa \nu^\sigma} &\ge 0 &\quad& \mbox{on } B_1\setminus \overline B_\va
\end{aligned}
\end{cases}
\]
for every $0<\va<1$. If $U\in C(\B_1^+ \cup B_1\setminus \{0\})$ and $U>0$ in $\B_1^+ \cup B_1 \setminus \{0\}$, then
\[
\liminf_{X\to 0}U(X)>0.
\]
\end{prop}

\section{Upper bound and lower bound near a singularity}\label{sec:bd}

In this section, we shall prove Theorem \ref{thm:a}, which will be used in the proof of asymptotical symmetry.

\begin{prop}\label{prop:up bd}
Suppose that $U$ is a nonnegative solution of \eqref{eq:ex0}. Then
\be \label{eq:clc}
\limsup_{|x|\to 0} |x|^{\frac{n-2\sigma}{2}}u(x)<\infty.
\ee
\end{prop}
\begin{proof} Suppose the contrary that there exists a sequence $\{x_j\} \subset B_1$ such that
\[
x_j\to 0\quad \mbox{as } j\to \infty,
\]
and
\be\label{eq:cl1}
|x_j|^{\frac{n-2\sigma}{2}}u(x_j)\to \infty\quad \mbox{as }j\to \infty.
\ee

Consider
\[
v_j(x):=\left(\frac{|x_j|}{2}-|x-x_j|\right)^{\frac{n-2\sigma}{2}} u(x),\quad |x-x_j|\leq \frac{|x_j|}{2}.
\]
Let $|\bar x_j-x_j|<\frac{|x_j|}{2}$ satisfy
\[
v_j(\bar x_j)=\max_{|x-x_j|\leq \frac{|x_j|}{2}}v_j(x),
\]
and let
\[
2\mu_j:=\frac{|x_j|}{2}-|\bar x_j-x_j|.
\]
Then
\be \label{eq:cl2}
0<2\mu_j\leq \frac{|x_j|}{2}\quad\mbox{and}\quad \frac{|x_j|}{2}-|x-x_j|\ge\mu_j \quad \forall ~ |x-\bar x_j|\leq \mu_j.
\ee
By the definition of $v_j$, we have
\be \label{eq:cl3}
(2\mu_j)^{\frac{n-2\sigma}{2}}u(\bar x_j)=v_j(\bar x)\ge v_j(x)\ge (\mu_j)^{\frac{n-2\sigma}{2}}u(x)\quad \forall ~ |x-\bar x_j|\leq \mu_j.
\ee
Thus, we have
\[
2^{\frac{n-2\sigma}{2}}u(\bar x_j)\ge u(x)\quad \forall ~ |x-\bar x_j|\leq \mu_j.
\]
We also have
\be\label{eq:cl4}
(2\mu_j)^{\frac{n-2\sigma}{2}}u(\bar x_j)=v_j(\bar x_j)\ge v(x_j)= \left(\frac{|x_j|}{2}\right)^{\frac{n-2\sigma}{2}}u(x_j)\to \infty\quad\mbox{as }i\to\infty.
\ee
Now, consider
\[
W_j(y,t)=\frac{1}{u(\bar x_j)}U\left(\bar x_j+\frac{y}{u(\bar x_j)^{\frac{2}{n-2\sigma}}}, \frac{t}{u(\bar x_j)^{\frac{2}{n-2\sigma}}}\right ), \quad (y,t)\in \om_j,
\]
where
\[
\om_j:=\left\{(y,t)\in \R^{n+1}_+| \left(\bar x_j+\frac{y}{u(\bar x_j)^{\frac{2}{n-2\sigma}}},\frac{t}{u(\bar x_j)^{\frac{2}{n-2\sigma}}}\right)\in  \B^+_{1}\setminus \{0\} \right\}.
\]
Let $w_j(y)=W_j(y,0)$. Then $W_j$ satisfies $w(0)=1$ and 
\be \label{eq:ext1}
\begin{cases}
\begin{aligned}
\mathrm{div}(t^{1-2\sigma} \nabla W_j)&=0 & \quad &\mbox{in }\om_j,\\
\frac{\pa W_j}{\pa \nu^\sigma}& = w_j^{\frac{n+2\sigma}{n-2\sigma}} &\quad &\mbox{on }\pa' \om_j.
\end{aligned}
\end{cases}
\ee
Moreover, it follows from \eqref{eq:cl3} and \eqref{eq:cl4} that 
\[
w_j(y)\leq 2^{\frac{n-2\sigma}{2}} \quad\mbox{in } B_{R_j},
\] 
where \[R_j:=\mu_j u(\bar x_j)^{\frac{2}{n-2\sigma}}\to \infty \mbox{ as } j\to \infty.\]

By Proposition \ref{prop:harnack}, for any given $\bar t>0$ we have
\[
0\leq W_j \leq C(\bar t)\quad \mbox{in }B_{R_j/2}\times [0,\bar t),
\]
where $C(\bar t)$ depends only on $n, \sigma$ and $\bar t$. 
Then by Corollary 2.1 and Theorem 2.7 in \cite{JLX} there exists some $\al>0$ such that for every $R>1$,
\[
\|W_j\|_{W^{1,2}(t^{1-2\sigma},\B^+_R)}+\|W_j\|_{C^\al(\overline \B^+_R)}+\|w_j\|_{C^{2,\al}(\overline B_R)}\le C(R),
\]
where $C(R)$ is independent of $i$. Thus, after passing to a subsequence, we have, for some nonnegative function
$W\in W^{1,2}_{loc}(t^{1-2\sigma},\overline{\mathbb{R}^{n+1}_+})\cap C^{\al}_{loc}(\overline{\mathbb{R}^{n+1}_+})$
\[
\begin{cases}
W_j&\rightharpoonup W\quad\mbox{weakly in }W^{1,2}_{loc}(t^{1-2\sigma},\R^{n+1}_+),\\
W_j&\rightarrow W\quad\mbox{in }C^{\al/2}_{loc}(\overline{\R^{n+1}_+}),\\
w_j&\rightarrow w\quad\mbox{in }C^2_{loc}(\R^n),
\end{cases}
\]
where $w(y)=W(y,0)$. Moreover, $W$ satisfies
\be \label{eq:ext2}
\begin{cases}
\begin{aligned}
\mathrm{div}(t^{1-2\sigma} \nabla W)&=0 & \quad &\mbox{in }\R^{n+1}_+,\\
\frac{\pa W}{\pa \nu^\sigma} &= w^{\frac{n+2\sigma}{n-2\sigma}}& \quad &\mbox{on }\pa\R^{n+1}_+,
\end{aligned}
\end{cases}
\ee
and $w(0)=1$.
By the Liouville theorem in \cite{JLX}, we have,
\be \label{eq:cl5}
w(y):=W(y,0)= \left(\frac{1}{1+|y|^2}\right)^\frac{n-2\sigma}{2},
\ee
upon some multiple, scaling and translation.

On the other hand, we are going to show that
\be\label{eq:aim1}
w_{\lda, x}(y)\leq w(y)\quad \forall~\lda>0, x\in \R^n, ~ |y-x|\ge\lda.
\ee
By an elementary calculus lemma in \cite{LZhang}, \eqref{eq:aim1} implies that $w\equiv constant$. This contradicts to \eqref{eq:cl5}.

Let us arbitrarily fix $x_0\in\R^n$ and $\lda_0>0$. Then for all $j$ large, we have $|x_0|<\frac{R_j}{10}, 0<\lda_0<\frac{R_j}{10}$. For $\lda>0$, we let 
\[
(W_j)_{X,\lda }(Y):= \left(\frac{\lda }{|Y-X|}\right)^{n-2\sigma}W_j\left(X+\frac{\lda^2(Y-X)}{|Y-X|^2}\right)
\]
for $Y\in \om_j$ with $|Y-X|\geq \lda$, which is the Kelvin transformation of $W_j$ with respect to the ball $\B_{X}(\lda)$. Let $X_0=(x_0,0)$.

\medskip

\emph{Claim 1}: There exists a positive real number $\lda_3$ such that for any $0<\lda<\lda_3$, we have
\[
(W_j)_{X_0,\lda }(\xi)\leq W_j(\xi) \quad \text{in } \om_j\backslash \B^+_{\lda}(X_0).
\]
The proof of Claim 1 consists of two steps as the proof of Lemma 3.2 in \cite{JLX}, which is inspired by \cite{CaL}.

\textit{Step 1.} We show that there exist $0<\lda_1<\lda_2<\lda_0$, which are independent on $j$, such that
\[
(W_j)_{X_0,\lda }(\xi)\leq W_j(\xi), ~\forall~0<\lda<\lda_1,~\lda<|\xi-X_0|<\lda_2.
\]
For every $0<\lda<\lda_1<\lda_2$, $\xi\in\pa'' \B_{\lda_2}(X_0)$, we have $X_0+\frac{\lda^2(\xi-X_0)}{|\xi-X_0|^2}\in \B^+_{\lda_2}(X_0)$. Thus we can choose $\lda_1=\lda_1(\lda_2)$ small such that
\begin{equation*}
\begin{split}
(W_j)_{X_0,\lda }(\xi)&=\left(\frac{\lambda}{|\xi-X_0|}\right)^{n-2\sigma}W_j\left(X_0+\frac{\lda^2(\xi-X_0)}{|\xi-X_0|^2}\right)\\
&\leq\left(\frac{\lda_1}{\lda_2}\right)^{n-2\sigma}\sup\limits_{\overline{\B_{\lda_2}^+(X_0)}}W_j\leq \inf_{\partial ''{\B_{\lda_2}^+(X_0)}}W_j\leq W_j(\xi),
\end{split}
\end{equation*}
where we used that $W_j$ converges to $W>0$ in $C^{\al/2}_{loc}(\overline{\R^{n+1}_+})$.
Hence
\[
(W_j)_{X_0,\lda }\leq W_j\quad \mbox{on } \partial ''(\B^+_{\lda_2}(X_0)\backslash \B^+_{\lda}(X_0))
\]
  for all $\lda_2>0$ and $0<\lda<\lda_1(\lda_2)$.

We will show that $(W_j)_{X_0,\lda }\leq W_j$ on $(\B^+_{\lda_2}(X_0)\backslash \B^+_{\lda}(X_0))$ if $\lda_2$ is small and $0<\lda<\lda_1(\lda_2)$.
Since $(W_j)_{X_0,\lda }$ also satisfies \eqref{eq:ext1} in $\B_{\lda_2}(X_0)^+\setminus\overline{\B_{\lda_1}^+(X_0)}$, we have
\begin{equation}\label{diff}
\begin{cases}
\mathrm{div}(t^{1-2\sigma}\nabla ((W_j)_{X_0,\lda }-W_j))&=0\quad \text{in}\quad \B_{\lda_2}^+(X_0)\backslash \overline{\B_{\lda}^+(X_0)},\\
\lim\limits_{t\to 0}t^{1-2\sigma}\partial_t ((W_j)_{X_0,\lda }-W_j)&\\ =
W_j^{\frac{n+2\sigma}{n-2\sigma}}(x,0)-(W_j)_{X_0,\lda }^{\frac{n+2\sigma}{n-2\sigma}}(x,0)
&\quad \text{on}\quad \partial '(\B_{\lda_2}^+(X_0)\backslash \overline{\B_{\lda}^+(X_0)}).
\end{cases}
\end{equation}
Let $((W_j)_{X_0,\lda }-W_j)^+:=\max(0,(W_j)_{X_0,\lda }-W_j)$ which equals to $0$ on $\pa''(\B^+_{\lda_2}(X_0)\backslash \B^+_{\lda}(X_0))$. Hence, by a density argument, we can use $((W_j)_{X_0,\lda }-W_j)^+$ as a test function in the definition of weak solution of \eqref{diff}. We will make use of the narrow domain technique
from \cite{BN}. With the help of the mean value theorem, we have
\begin{equation*}
\begin{split}
&\int_{\B_{\lda_2}^+(X_0)\backslash \B_{\lda}^+(X_0)} t^{1-2\sigma}|\nabla((W_j)_{X_0,\lda }-W_j)^+|^2\\
&=\int_{B_{\lda_2}(X_0)\backslash B_{\lda}(X_0)}((W_j)_{X_0,\lda }^{\frac{n+2\sigma}{n-2\sigma}}(x,0)-W_j^{\frac{n+2\sigma}{n-2\sigma}}(x,0))((W_j)_{X_0,\lda }-W_j)^+\\
&\leq C \int_{B_{\lda_2}(X_0)\backslash B_{\lda}(X_0)}(((W_j)_{X_0,\lda }-W_j)^+)^2 (W_j)_{X_0,\lda }^{\frac{4\sigma}{n-2\sigma}}\\
&\leq C\left(\int_{B_{\lda_2}(X_0)\backslash B_{\lda}(X_0)}(((W_j)_{X_0,\lda }-W_j)^+)^\frac{2n}{n-2\sigma}\right)^{\frac{n-2\sigma}{n}}\left(\int_{B_{\lda_2}(X_0)\backslash B_{\lda}(X_0)} (W_j)_{X_0,\lda }^{\frac{2n}{n-2\sigma}}\right)^{\frac{2\sigma}{n}}\\
& \leq C\left(\int_{\B_{\lda_2}^+(X_0)\backslash \B_{\lda}^+(X_0)} t^{1-2\sigma}|\nabla((W_j)_{X_0,\lda }-W_j)^+|^2\right)
\left(\int_{ B_{\lda_2}(X_0)} w_j^{\frac{2n}{n-2\sigma}}\right)^{\frac{2\sigma}{n}},
\end{split}
\end{equation*}
where Proposition 2.1 in \cite{JLX} is used in the last inequality and $C$ is a positive constant depending only on $n$ and $\sigma$. Since $w_j\to w$ in $C^2(B_{\lda_0}(X_0))$ as in \eqref{eq:cl4},
we can fix $\lda_2$ small independent of $j$ such that
\[
C\left(\int_{B_{\lda_2}(X_0)} w_j^{\frac{2n}{n-2\sigma}}\right)^{\frac{2\sigma}{n}}<1/2.
\]
Then 
\[
\nabla((W_j)_{X_0,\lda }-W_j)^+=0\quad \mbox{in } \B_{\lda_2}^+(X_0)\backslash \B_{\lda}^+(X_0).
\]
Since 
\[
((W_j)_{X_0,\lda }-W_j)^+=0 \quad\mbox{on }\partial ''(\B^+_{\lda_2}(X_0)\backslash \B^+_{\lda}(X_0)),\]
 we have 
 \[
 ((W_j)_{X_0,\lda }-W_j)^+=0\quad \mbox{in } \B_{\lda_2}^+(X_0)\backslash \B_{\lda}^+(X_0).
 \]
 We conclude that 
 \[
 (W_j)_{X_0,\lda }\leq W_j \quad\mbox{in } \B^+_{\lda_2}(X_0)\backslash \B^+_{\lda}(X_0)
 \] 
 for $0<\lda<\lda_1:=\lda_1(\lda_2)$.
\medskip

\textit{Step 2.}  We show that there exists $\lda_3\in (0,\lda_1)$ such that $\forall~0<\lda<\lda_3$,
\[
 (W_j)_{X_0,\lda }(\xi)\leq W_j(\xi),~\forall |\xi-X_0|>\lda_2,~\xi\in \om_j.
\]
Let $\phi_j(\xi)=\left(\frac{\lda_2}{|\xi-X_0|}\right)^{n-2\sigma}\inf\limits_{\partial'' \B_{\lda_2}(X_0)} W_j$,
which satisfies
\[
\begin{cases}
\begin{aligned}
  \mathrm{div}(t^{1-2\sigma}\nabla \phi_j)&=0&\quad& \mbox{in } \R^{n+1}_+\setminus \B_{\lda_2}^+(X_0)\\
   -\lim_{t\to 0}t^{1-2\sigma}\pa_t \phi_j(x,0)&=0&\quad& \mbox{on } \R^n\setminus \overline{B_{\lda_2}(X_0)},
 \end{aligned}
\end{cases}
\]
 and $\phi_j(\xi)\leq W_j(\xi)$ on $\partial'' \B_{\lda_2}(X_0)$. Let us examine them on $\pa''\om_j$.
 
 Since $u\ge 1/C>0$ on $\pa B_1$, it follows from the Harnack inequality (Proposition \ref{prop:harnack}) that
\be\label{eq:cl20}
W_j\geq \frac{1}{Cu(\bar x_j)}>0 \quad \mbox{on }\pa'' \om_j.
\ee
 Since $\frac{|x_j|}{2}\le |\bar x_j|\leq \frac{3|x_j|}{2}<<1$, for any $\xi\in \pa'' \om_j$, i.e., $\left|\bar X_j+\frac{\xi}{u(\bar x_j)^{\frac{2}{n-2\sigma}}}\right|=1$, we have
\[
|\xi|\approx u(\bar x_j)^{\frac{2}{n-2\sigma}}.
\]
 Thus
 \be\label{eq:cl21}
W_j(\xi)\geq \frac{1}{Cu(\bar x_j)}>\left(\frac{\lda_2}{|\xi-X_0|}\right)^{n-2\sigma}\inf\limits_{\partial'' \B_{\lda_2}(X_0)} W_j \quad \mbox{on }\pa'' \om_j,
\ee
where we used the fact that $W_j$ converges to a solution $W$ of \eqref{eq:ext2} locally uniformly in the last inequality. By the maximum principle,
\be\label{eq:cl22}
W_j(\xi)\geq \left(\frac{\lda_2}{|\xi-X_0|}\right)^{n-2\sigma}\inf\limits_{\partial'' \B_{\lda_2}(X_0)} W_j,~\forall~|\xi-X_0|>\lda_2,~\xi\in \om_j.
\ee
Let 
\[
\lda_3=\min(\lda_1, \lda_2(\inf\limits_{\partial'' \B_{\lda_2}(X_0)} W_j/\sup\limits_{ \B_{\lda_2}(X_0)} W_j)^{\frac{1}{n-2\sigma}}).
\]
Then for any $0<\lda<\lda_3,~|\xi-X_0|\geq \lda_2$, $\xi\in\om_j$, we have
\[
\begin{split}
 (W_j)_{X_0,\lda }(\xi)&\leq (\frac{\lda}{|\xi-X_0|})^{n-2\sigma}W_j(X_0+\frac{\lda^2(\xi-X_0)}{|\xi-X_0|^2})\\
&\leq (\frac{\lda_3}{|\xi-X_0|})^{n-2\sigma}\sup\limits_{\B_{\lda_2}(X_0)}W_j\\
&\leq (\frac{\lda_2}{|\xi-X_0|})^{n-2\sigma}\inf\limits_{\partial ''\B_{\lda_2}(X_0)}W_j\leq W_j(\xi).
\end{split}
\]
Claim 1 is proved.

We define
\[
\bar \lda:=\sup \{0<\mu\le \lda_0 | (W_j)_{X_0,\lda}(\xi)\leq W_{j}(\xi),\quad \forall~|\xi-X_0|\geq \lda, ~\xi\in \om_j,~\forall~ 0<\lda <\mu\},
\]
where $\lda_0$ and $X_0$ are fixed at the beginning. By Claim 1, $\bar\lda$ is well defined.

\medskip

\emph{Claim 2}: $\bar\lda=\lda_0.$

\medskip

To prove Claim 2, we argue by contradiction. Suppose $\bar\lda<\lda_0$. Similar to \eqref{eq:cl21}, we have that
\[
W_j(\xi)\geq \frac{1}{Cu(\bar x_j)}>(\frac{\lda_0}{|\xi-X_0|})^{n-2\sigma}\sup\limits_{\B_{\lda_0}(X_0)}W_j\ge(W_j)_{X_0,\bar\lda}(\xi) \quad \mbox{on }\pa'' \om_j.
\]
It follows from strong maximum principle that $(W_j)_{X_0,\bar\lda}(\xi)< W_{j}(\xi)$ if $|\xi-X_0|> \bar\lda, ~\xi\in \overline \om_j$. For $\delta>0$ small, which will be fixed later, denote $K_{\delta}=\{\xi\in\om_j: |\xi-X_0|\geq\bar\lda+\delta\}$.
By Proposition \ref{prop:liminf}, there exists $c_3=c_3(\delta)>0$ such that
\[
W_{j}(\xi)-(W_j)_{X_0,\bar\lda}(\xi)>c_3  \ \text{ in }\  K_{\delta}.
\]
By the uniform continuity of $W_j$ on compact sets, there exists $\va$ small such that for all $\bar\lda<\lda<\bar\lda+\va$
\[
(W_j)_{X_0,\bar\lda}-(W_j)_{X_0,\lda}>-c_3/2 \ \text{ in }\  K_{\delta}.
\]
Hence
\[
W_{j}-(W_j)_{X_0,\lda}>c_3/2 \ \text{ in } \ K_{\delta}.
\]
Now let us focus on the region $\{\xi\in\R^{n+1}_+: \lda\leq|\xi-X_0|\leq\bar\lda+\delta\}$.
Using the narrow domain technique as that in Claim 1, we can choose $\delta$
small (notice that we can choose $\va$ as small as we want) such that
\[
W_{j}\geq (W_j)_{X_0,\lda} \ \text{ in } \ \{\xi\in\R^{n+1}_+: \lda\leq|\xi-X_0|\leq\bar\lda+\delta\}.
\]
In conclusion, there exists $\va_1>0$ such that for all $\bar\lda<\lda<\bar\lda+\va_1$
\[
 (W_j)_{X_0,\lda}(\xi)\leq W_{j}(\xi),\quad \forall~|\xi-X_0|\geq \lda, ~\xi\in \overline \om_j,
\]
which contradicts with the definition of $\bar\lda$. Claim 2 is proved.

Thus 
\[
 (W_j)_{X_0,\lda}(\xi)\leq W_{j}(\xi),\quad \forall~|\xi-X_0|\geq \lda,~\xi\in \overline \om_j,~\forall~0<\lda\leq \lda_0.
\]
Sending $j\to \infty$, we have
\[
w_{x_0,\lda}(y)\leq w(y)\quad \forall~0<\lda\leq \lda_0,  ~ |y-x_0|\ge\lda.
\]
Since $x_0, \lda_0$ are arbitrary, \eqref{eq:aim1} has been verified. 

The proposition is proved.
\end{proof}

We remark that the above arguments also apply to subcritical cases.

One consequence of this upper bound is that every solution $U$ of \eqref{eq:ex0} satisfies the following the Harnack inequality, which will be used very frequently in this rest of the paper. 
\begin{lem}\label{lem:spherical harnack}
Suppose that $U$ is a nonnegative solution of \eqref{eq:ex0}. Then for all $0<r<1/4$, we have
\be\label{eq:spherical harnack}
\sup_{\B^+_{2r}\setminus\overline{\B^+_{r/2}}} U\le C \inf_{\B^+_{2r}\setminus\overline{\B^+_{r/2}}} U,
\ee
where $C$ is a positive constant independent of $r$. 
\end{lem}
\begin{proof}
Let
\[
V(X)=r^{\frac{n-2\sigma}{2}}U(rX).
\]
It follows from Proposition \ref{prop:up bd} that
\[
|V(x,0)|\le C\quad\mbox{for all }1/4\le |x|\le 4,
\]
where $C$ is a positive constant depending on $U$ but independent of $r$. Moreover, $V$ satisfies \eqref{eq:ex0} as well. By the Harnack inequality in Proposition \ref{prop:harnack} and the standard Harnack inequality for uniformly elliptic equations, we have
\[
\sup_{1/2\le|X|\le 2} V(X)\le C \inf_{1/2\le|X|\le 2} V(X),
\]
where $C$ is another positive constant independent of $r$. Hence, \eqref{eq:spherical harnack} follows.
\end{proof}

By the Harnack inequality \eqref{eq:spherical harnack} we actually have
\be\label{eq:infinity}
\liminf_{|(x,t)|\to 0}U(x,t)=\infty
\ee
if $0$ is a non-removable singularity of $U$. We know that there exists a sequence of points $\{x_j\}$ such that 
\[
r_j=|x_j|\to 0\quad \mbox{and}\quad U(x_j, 0)\to\infty\quad\mbox{as }j\to\infty.
\]
It follows from \eqref{eq:spherical harnack} that
\[
\inf_{|X|=r_j} U(X)\ge C^{-1}  U(x_j,0).
\]
By the maximum principle,
\[
\inf_{r_{j+1}\le |\xi|\le r_j} U(\xi)= \inf_{|\xi|=r_j,\ r_{j+1}} U(\xi) \ge C^{-1} \min(U(x_j,0), U(x_{j+1},0))\to\infty
\quad\mbox{as }j\to\infty.
\]
The claim is proved.

To prove the lower bound in \eqref{eq:low and up bound}, we will make use of a Pohozaev identity. For a nonnegative solution $U$ of \eqref{eq:ex0}, we define the Pohozaev integral as
\[
\begin{split}
P(U,R)&=\frac{n-2\sigma}{2}\int_{\pa'' \B_{R}}t^{1-2\sigma}\frac{\pa U}{\pa\nu} U
-\frac{R}{2}\int_{\pa'' \B_{R}}t^{1-2\sigma}|\nabla U|^2\\
&\quad+R\int_{\pa'' \B_{R}}t^{1-2\sigma}|\frac{\pa U}{\pa\nu}|^2
+\frac{n-2\sigma}{2n}R\int_{\pa B_{R}}u^{\frac{2n}{n-2\sigma}},
\end{split}
\]
where $u(\cdot)=U(\cdot, 0)$. By the Pohozaev identity (see, e.g., the proof of Proposition 4.3 in \cite{JLX}), $P(U,R)$ is independent of $R$, and we denote it as $P(U)$.

\begin{prop} \label{prop:low bd}
Let $U$ be a nonnegative solution of \eqref{eq:ex0}. 
If 
\[
\liminf_{|x|\to 0}|x|^{\frac{n-2\sigma}{2}}u(x)=0,
\]
then
\[
\lim_{|x|\to 0}|x|^{\frac{n-2\sigma}{2}}u(x)=0.
\]
\end{prop}

\begin{proof}

We suppose by contradiction that
\[
\liminf_{|x|\to 0}|x|^{\frac{n-2\sigma}{2}}u(x)=0\quad\mbox{and}\quad\limsup_{|x|\to 0}|x|^{\frac{n-2\sigma}{2}}u(x)=C>0.
\]
Hence, there exist two sequences of points $\{x_i\}, \{y_i\}$ satisfying
\[
x_i\to 0,\quad y_i\to 0\quad \mbox{as }i\to\infty,
\]
such that
\[
|x_i|^{\frac{n-2\sigma}{2}}u(x_i)\to 0\quad\mbox{and}\quad |y_i|^{\frac{n-2\sigma}{2}}u(y_i)\to C>0\quad \mbox{as }i\to\infty.
\]
Then there exists a sequence of positive numbers $\{r_i\}$ converging to $0$ such that
\[
r_i^{\frac{n-2\sigma}{2}}\bar u(r_i)\to 0\quad\mbox{as }i\to\infty\quad\mbox{and}\quad r_i\mbox{ are local minimum of }r^{\frac{n-2\sigma}{2}}\bar u(r)\quad\mbox{for every }i,
\]
where $\bar u(r)$ is the spherical average of $u$ on $\pa B_r$.
Let 
\[
W_i(X)=\frac{U(r_iX)}{U(r_ie_1)},
\]
where $e_1=(1,0,\cdots, 0)$. Then $W_i(X)$ is locally uniformly bounded away from the origin, which follows from the Harnack inequality \eqref{eq:spherical harnack}, and satisfies
\[
\begin{cases}
\begin{aligned}
\mathrm{div}(t^{1-2\sigma} \nabla W_i)&=0 & \quad &\mbox{in }\R^{n+1}_+,\\
\frac{\pa W_i}{\pa \nu^\sigma} (x,0)&=( r_i^{\frac{n-2\sigma}{2}}U(r_ie_1))^\frac{4\sigma}{n-2\sigma}W_i^{\frac{n+2\sigma}{n-2\sigma}}&\quad &\mbox{on }\R^n\setminus\{0\}.
\end{aligned}
\end{cases}
\]
Notice that by the Harnack inequality \eqref{eq:spherical harnack}, $r_i^{\frac{n-2\sigma}{2}}U(r_ie_1)\to 0$ as $i\to\infty$. By Corollary 2.1 in \cite{JLX} and Theorem 2.7 in \cite{JLX} there exists some $\al>0$ such that for every $R>1>r>0$,
\[
\|W_j\|_{W^{1,2}(t^{1-2\sigma},\B^+_R\setminus\overline \B^+_r)}+\|W_j\|_{C^\al( \B^+_R\setminus\overline \B^+_r)}+\|w_j\|_{C^{2,\al}( B_R\setminus\overline \B^+_r)}\le C(R,r),
\]
where $C(R,r)$ is independent of $i$. Then up to a subsequence, $\{W_i\}$ converges to a nonnegative function $W\in W^{1,2}_{loc}(t^{1-2\sigma},\overline{\mathbb{R}^{n+1}_+}\setminus\{0\})\cap C^{\al}_{loc}(\overline{\mathbb{R}^{n+1}_+}\setminus\{0\})$ satisfying
\[
\begin{cases}
\begin{aligned}
\mathrm{div}(t^{1-2\sigma} \nabla W)&=0 & \quad &\mbox{in }\R^{n+1}_+,\\
\frac{\pa W}{\pa \nu^\sigma} (x,0)&= 0&\quad &\mbox{on }\R^n\setminus\{0\}.
\end{aligned}
\end{cases}
\]
By a B\^ocher type theorem in Lemma 4.4 in \cite{JLX}, we have
\[
W(X)=\frac{a}{|X|^{n-2\sigma}}+b,
\]
where $a,b$ are nonnegative constants. Let $w(x)=W(x,0)$. We know that $w_i(x)\to w(x)$ in $C^2_{loc}(\R^n\setminus\{0\})$. Thus, $r^{\frac{n-2\sigma}{2}}\bar w(r)$ has a critical point at $r=1$, which implies that $a=b$. Since $W(e_1)=1$, we have $a=b=1/2$. Now let us compute $P(U)$.

It follows from Proposition 2.6 in \cite{JLX} that $|\nabla_x W_i|$ and $|t^{1-2\sigma}\pa_t W_i|$ are locally uniformly bounded in $C_{loc}^{\gamma}(\overline{\R^{n+1}_+}\setminus\{0\})$ for some $\gamma>0$. We have, for some $C>0$,
\[
|\nabla_x U(X)|\le C r_i^{-1} U(r_ie_1)=o(1)r_i^{-\frac{n-2\sigma}{2}-1}\quad\mbox{for all }|X|=r_i
\]
and
\[
|t^{1-2\sigma} U_t(X)|\le C r_i^{-2\sigma} U(r_ie_1)=o(1)r_i^{-\frac{n-2\sigma}{2}-2\sigma}\quad\mbox{for all }|X|=r_i.
\]
Thus
\[
P(U)=\lim_{i\to\infty} P(U,r_i)=0.
\]
Hence
\[
P(U,r_i)=0\quad\mbox{for all }i.
\]

On the other hand, for all $i$, 
\[
0=P(U,r_i)=P(r_i^{\frac{n-2\sigma}{2}}U(r_iX), 1)=P(r_i^{\frac{n-2\sigma}{2}}U(r_ie_1)W_i,1).
\]
Hence, we have
\[
\begin{split}
0&=\frac{n-2\sigma}{2}\int_{\pa'' \B_{1}}t^{1-2\sigma}\frac{\pa W_i}{\pa\nu} W_i
-\frac{1}{2}\int_{\pa'' \B_{1}}t^{1-2\sigma}|\nabla W_i|^2\\
&\quad+\int_{\pa'' \B_{1}}t^{1-2\sigma}|\frac{\pa W_i}{\pa\nu}|^2
+\frac{n-2\sigma}{2n}\int_{\pa B_{1}}(r_i^{\frac{n-2\sigma}{2}}U(r_ie_1))^{\frac{4\sigma}{n-2\sigma}}W_i^{\frac{2n}{n-2\sigma}}.
\end{split}
\]
Sending $i\to\infty$, we have
\[
\begin{split}
0&=\frac{n-2\sigma}{2}\int_{\pa'' \B_{1}}t^{1-2\sigma}\frac{\pa W}{\pa\nu} W
-\frac{1}{2}\int_{\pa'' \B_{1}}t^{1-2\sigma}|\nabla W|^2+\int_{\pa'' \B_{1}}t^{1-2\sigma}|\frac{\pa W}{\pa\nu}|^2\\
&=-\frac{(n-2\sigma)^2}{8}\int_{\pa'' \B_{1}}t^{1-2\sigma},
\end{split}
\]
which is a contradiction.
\end{proof}

\begin{prop}\label{prop:continuous}
Let $U$ be a nonnegative solution of \eqref{eq:ex0}. 
If 
\[
\lim_{|x|\to 0}|x|^{\frac{n-2\sigma}{2}}u(x)=0,
\]
then $u$ can be extended as a continuous function near the origin $0$.
\end{prop}
\begin{proof}
By the Harnack inequality \eqref{eq:spherical harnack}, we have $ \lim_{|X|\to 0}|X|^{\frac{n-2\sigma}{2}}U(X)=0$. For $0<\mu\le n-2\sigma$ and $\delta>0$, let
\[
\Phi_\mu(X):=|X|^{-\mu}-\delta t^{2\sigma}|X|^{-(\mu+2\sigma)}.
\]
Then we have
\[
\mathrm{div}(t^{1-2\sigma}\nabla \Phi_\mu(X))=t^{1-2\sigma}|X|^{-(\mu+2)}\left(-\mu(n-2\sigma-\mu)+\frac{\delta(\mu+2\sigma)(n-\mu)t^{2\sigma}}{|X|^{2\sigma}}\right),
\]
and
\[
-\dlim_{t\rightarrow 0^+}t^{1-2\sigma}\pa_t \Phi_\mu(x,s)=2\delta\sigma|x|^{-(\mu+2\sigma)}= 2\delta\sigma|x|^{-2\sigma}\Phi_{\mu}(x,0).
\]
Let $\alpha\in (0,\frac{n-2\sigma}{2})$ be fixed, $\beta=\frac{n-2\sigma}{2}+1$ and $\Phi=C\Phi_\alpha+\va\Phi_\beta$, where $C, \va$ are positive constants. We can choose $\delta$ small (depending on $\al$) such that
\[
\begin{cases}
\begin{aligned}
\mathrm{div}(t^{1-2\sigma} \nabla \Phi)&\le 0 & \quad &\mbox{in }\B^+_2,\\
\frac{\pa \Phi}{\pa \nu^\sigma} (x,0)&= 2\delta\sigma|x|^{-2\sigma}\Phi(x,0)&\quad &\mbox{on }B_2\setminus\{0\}.
\end{aligned}
\end{cases}
\]
Let $\tau$ be such that $a(x)=u^{\frac{4\sigma}{n-2\sigma}}\le 2\delta\sigma|x|^{-2\sigma}$ for all $0<|x|<\tau$. Then we have
\[
\begin{cases}
\begin{aligned}
\mathrm{div}(t^{1-2\sigma} \nabla (\Phi-U))&\le 0 & \quad &\mbox{in }\B^+_\tau,\\
\frac{\pa (\Phi-U)}{\pa \nu^\sigma} (x,0)&\ge 2\delta\sigma|x|^{-2\sigma}(\Phi(x,0)-U(x,0))&\quad &\mbox{on }B_\tau\setminus\{0\}.
\end{aligned}
\end{cases}
\]
For every $\va>0$, we have that $\Phi\ge U$ near $0$. We can choose $C$ (depending on $\al$) sufficiently large so that $\Phi\ge U$ on $\pa'' \B_\tau$. Hence, by the maximum principle in Lemma A.1 in \cite{JLX} (we can choose $\delta$ even smaller if needed), we have
\[
\Phi\ge U \quad\mbox{in } \B_\tau^+\setminus\{0\}.
\]
By sending $\va\to 0$, we have
\[
 U\le  C(\al)\Phi_\alpha\le C(\al)|X|^{-\al}\quad\mbox{in } \B_\tau^+\setminus\{0\}.
\]
It follows from standard rescaling arguments, with the help of Proposition 2.6 in \cite{JLX} and standard uniform elliptic equations theory, that
\[
 |\nabla_x U(X)|\le C(\al)|X|^{-\al-1}\quad\mbox{in } \B_\tau^+\setminus\{0\}
\]
and
\[
 |t^{1-2\sigma}\pa_t U(X)|\le C(\al)|X|^{-\al-2\sigma}\quad\mbox{in } \B_\tau^+\setminus\{0\}.
\]
Since $\al<\frac{n-2\sigma}{2}$, it is elementary to verify that $U\in W^{1,2}(t^{1-2\sigma},\B^+_\tau)$. Moreover, $U$ satisfies
\[
\begin{cases}
\begin{aligned}
\mathrm{div}(t^{1-2\sigma} \nabla U)&=0 & \quad &\mbox{in }\B_\tau^+,\\
\frac{\pa U}{\pa \nu^\sigma} (x,0)&= U^{\frac{n+2\sigma}{n-2\sigma}}(x,0) &\quad& \mbox{on }\pa' \B_\tau.
\end{aligned}
\end{cases}
\]
Indeed, for $\va>0$ small, let $\eta_\va$ be a smooth cut-off function satisfying $\eta\equiv 0$ in $\B_\va$, $\eta\equiv 1$ outside of $\B_{2\va}$ and $\nabla \eta_\va\le C\va^{-1}$. Let $\varphi\in C_c^{\infty}(\B_\tau^+\cup\pa'\B_\tau^+)$. It follows that
\[
\int_{\B_\tau^+} t^{1-2\sigma}\nabla U\nabla (\varphi\eta_\va)=\int_{\pa'\B_\tau^+} U^{\frac{n+2\sigma}{n-2\sigma}}(x,0)\varphi\eta_\va.
\]
By the dominated convergence theorem and sending $\va\to 0$, we have
\[
\int_{\B_\tau^+} t^{1-2\sigma}\nabla U\nabla \varphi=\int_{\pa'\B_\tau^+} U^{\frac{n+2\sigma}{n-2\sigma}}(x,0)\varphi.
\]

Finally, since $U(\cdot,0)\in L^p(B_1)$ for some $p>\frac{n}{2\sigma}$, it follows from Proposition 2.4 in \cite{JLX} that $U$ is H\"older continuous in $\overline \B_{\tau/2}^+$.
\end{proof}

\begin{proof}[Proof of Theorem \ref{thm:a}]
It follows from Propositions \ref{prop:up bd}, \ref{prop:low bd} and \ref{prop:continuous}.
\end{proof}

\section{Global solutions with an isolated singularity}\label{sec:global}

\begin{proof}[Proof of Theorem \ref{thm:gs}]
It follows from Proposition \ref{prop:liminf} that
\[
\liminf_{|\xi|\to 0}U(\xi)>0.
\]
First, we would like to show that for all $x\in\R^n\setminus\{0\}$ there exists $\lda_3(x)\in (0,|x|)$ such that for all $0<\lda<\lda_3(x)$ we have
\be\label{eq:bigger}
U_{X,\lda}(\xi)\le U(\xi)\quad\forall\ |\xi-X|\ge \lda,~\xi\neq 0,
\ee
where $X=(x,0)$ and
\[
U_{X,\lda }(Y):= \left(\frac{\lda }{|Y-X|}\right)^{n-2\sigma}U\left(X+\frac{\lda^2(Y-X)}{|Y-X|^2}\right).
\]
This can be proved similarly to that for $W_j$ in the proof of Theorem \ref{thm:a}, and we sketch the proofs here.
The first step is to show that there exist $0<\lda_1<\lda_2<|x|$ such that
\[
U_{X,\lda }(\xi)\leq U(\xi), ~\forall~0<\lda<\lda_1,~\lda<|\xi-X|<\lda_2.
\]
The proof of this step follows exactly the same as that for $W_j$ before. The second step is to show that there exists $\lda_3(x)\in (0,|x|)$ such that \eqref{eq:bigger} holds for all $0<\lda<\lda_3(x)$. To prove this step, we only need to make sure that \eqref{eq:cl21} holds for $U$, i.e.,
\be\label{eq:cl22U}
U(\xi)\geq \left(\frac{\lda_2}{|\xi-X|}\right)^{n-2\sigma}\inf_{\partial'' \B_{\lda_2}(X)} U,~\forall~|\xi-X|>\lda_2,~\xi\neq 0,
\ee
where $\lda_2<|x|$ is small. And \eqref{eq:cl22U} follows from a standard maximum principle argument.

Now, we can define
\[
\bar \lda(x):=\sup \{0<\mu\le |x|\ |\ U_{X,\lda}(\xi)\leq U(\xi),\quad \forall~|\xi-X|\geq \lda, ~\xi\neq 0,~\forall~ 0<\lda <\mu\}.
\]
Secondly, we will show that
\be\label{eq:gseq}
\bar \lda(x)=|x|.
\ee
Suppose $\bar \lda(x)<|x|$ for some $x\neq 0$. Since $0$ is not removable, by strong maximum principle we have $U(\xi)>U_{X,\lda}(\xi)$ for $|\xi-X|> \lda, ~\xi\neq 0$. It follows from Proposition \ref{prop:liminf} that
\[
\liminf_{\xi\to 0} (U(\xi)-U_{X,\lda}(\xi))>0.
\]
Then using the narrow domain technique as before (see also the proof of Theorem 1.5 in \cite{JLX}), the moving sphere procedure may continue beyond $\bar \lda(x)$ where we reach a contradiction. This proved \eqref{eq:gseq}. Thus
\be\label{eq:gseq2}
\ U_{X,\lda}(\xi)\leq U(\xi),\quad \forall~|\xi-X|\geq \lda, ~\xi\neq 0,~\forall~ 0<\lda <|x|.
\ee
For any unit vector $e\in\R^n$, for any $a>0$, $\xi=(y,t)\in\overline{\R^{n+1}_+}$ satisfying $(\xi-ae)\cdot e<0$, \eqref{eq:gseq2} holds with $x=Re$ and $\lda=R-a$. Sending $R$ to infinity, we have
\[
U(y,t)\ge U(y-2(y\cdot e-a)e,t).
\]
This shows the radial symmetry and non-increasing property of $u$ in $r$. Since we can differentiate the equation \eqref{eq:ex01} w.r.t. $x$ (see Proposition 2.5 in \cite{JLX}), then by applying the Harnack inequality in Proposition \ref{prop:harnack} to the equation of $U_r$, we have $U_r<0$. Theorem \ref{thm:gs} is proved. 
\end{proof}

\section{Asymptotical radial symmetry}\label{sec:asymptotical}

\begin{proof}[Proof of Theorem \ref{thm:b}]
As before, we have that for all $0<|x|<\frac 14$, $X=(x,0)$,
\[
\bar \lda(x):=\sup \{0<\mu\le |x|\ |\ U_{X,\lda}(\xi)\leq U(\xi),\quad \forall~|\xi-X|\geq \lda, ~0<|\xi|\leq 1,~\forall~ 0<\lda <\mu\}
\]
is well-defined and $\bar \lda(x)>0$, where we denote $\xi=(y,t)$. We are not going to prove this statement, since its proof is very similar to those in the previous two sections. One only need to notice that we can choose $\lda_2$ small such that
\be\label{eq:as}
U(\xi)\geq \left(\frac{\lda_2}{|\xi-X|}\right)^{n-2\sigma}\inf_{\partial'' \B_{\lda_2}(X)} U,\quad\forall~\xi\in\pa''\B^+,
\ee
which implies
\be\label{eq:as2}
U(\xi)\geq \left(\frac{\lda_2}{|\xi-X|}\right)^{n-2\sigma}\inf_{\partial'' \B_{\lda_2}(X)} U,\quad\forall~|\xi-X|>\lda, 0<|\xi|\le 1.
\ee
For $y\in B_2$, $\frac 34\le|y|\le \frac 54$ and $0<\lda<|x|<\frac 14$,
\[
\left|x+\frac{\lda^2(y-x)}{|y-x|^2}-x\right|\le 2\lda^2\le 2|x|^2\le \frac{|x|}{2}.
\]
it follows from Theorem \ref{thm:a} that
\[
u\left(x+\frac{\lda^2(y-x)}{|y-x|^2}\right)\le C|x|^{\frac{2\sigma-n}{2}}.
\]
Thus,
\[
u_{x,\lda}(y)=U_{X,\lda}(y,0)=C\lda^{n-2\sigma}|x|^{\frac{2\sigma-n}{2}}\le C|x|^{\frac{n-2\sigma}{2}}\quad\forall 0<\lda<|x|<\frac 14,\ \frac 34\le|y|\le \frac 54.
\]
By Harnack inequality in Proposition \ref{prop:harnack}, for all $|\xi|=1$, we have
\[
U_{X,\lda}(\xi)\le C|x|^{\frac{n-2\sigma}{2}}<U(\xi)\quad\forall 0<\lda<|x|\le \va/2,\ |\xi|=1
\]
for $\va>0$ sufficiently small. Moreover, it follows from Proposition \ref{prop:liminf} that
\[
\liminf_{\xi\to 0} (U(\xi)-U_{X,\lda}(\xi))>0.
\]
As before, given these two properties with narrow domain techniques, the moving sphere procedure may continue if $\bar\lda(x)< |x|$. Thus we obtain $\bar\lda(x)=|x|$ for $|x|\le\va/2$, where $\va$ is sufficiently small. Thus, we have proved that there exists some constant $\va>0$ such that
\be\label{eq:small stop}
U_{X,\lda}(\xi)\le U(\xi)\quad \forall~0<\lda<|x|\le \va/2,~|\xi-X|\ge\lda,~0<|\xi|\le 1.
\ee
In particular
\be\label{eq:small stop1}
u_{x,\lda}(y)\le u(y)\quad \forall~0<\lda<|x|\le \va/2,~|y-x|\ge\lda,~0<|y|\le 1.
\ee
By Lemma A.2 in \cite{LL} (or more precisely, its proof there), we have
\be\label{eq:ln est}
|\nabla \log u(x)|\le \frac{n-2\sigma}{|x|}\quad\mbox{for all }0<|x|<\frac{\va}{4}. 
\ee
Indeed, for $x\in B_{\va/4}\setminus\{0\}$, let $z=x-se$ where $0<s<\frac{|x|}{2}$ and $e\in\mathbb{S}^n$. It follows from \eqref{eq:small stop1} that
\[
u_{z,s}(y)\le u(y)\quad \forall~|y-z|\ge s,~0<|y|\le 1.
\]
Let $y=z+\tilde s e$ for some $\tilde s>s$ but close to $s$. Then we have
\[
\left(\frac{s^2}{\tilde s}\right)^{\frac{n-2\sigma}{2}}u\left(z+\frac{s^2}{\tilde s}e\right)\le \tilde s ^{\frac{n-2\sigma}{2}}u(z+\tilde se),
\]
which implies that
\[
\left.\frac{\ud }{\ud h}\right|_{h=s}h^{\frac{n-2\sigma}{2}}u(z+he)\ge 0.
\]
Consequently,
\[
-\nabla u(x)\cdot e\le \frac{n-2\sigma}{2s}u(x).
\]
By sending $s\to |x|/2$, we have 
\[
|\nabla u(x)|\le \frac{n-2\sigma}{|x|}u(x),
\]
which proves \eqref{eq:ln est}. Let
\[
V(\xi)=|\xi|^{2\sigma-n}U\left(\frac{\xi}{|\xi|^2}\right)\quad\mbox{and}\quad v(\cdot)=V(\cdot,0).
\]
Then it follows from \eqref{eq:small stop1} that for all $\mu>M:=\frac {1}{\va}$,
\[
v(y)\le v(y_{\mu})\quad\forall \ y\cdot e\ge\mu,~|y_\mu|\ge 1,~e\in\R^n,~|e|=1,
\]
where $y_\mu=y+2(\mu-y\cdot e)e$ is the reflection of $y$ with respect to the plane $x\cdot e=\mu$.
Thus, there exists $c>0$ independent of $M$ such that
\[
v(x)\ge v(y)\quad\mbox{whenever }|x|>1,~|y|\ge |x|+cM.
\]
It follows that for $R$ large
\[
\sup_{|x|=R+cM} v(x)\le \inf_{|x|=R}v(x)\le\sup_{|x|=R}v(x)\le\inf_{|x|=R-cM}v(x).
\]
Since
\[
u(x)=\left(\frac{1}{|x|}\right)^{n-2\sigma}v\left(\frac{x}{|x|^2}\right),
\]
It follows that
\be\label{eq:compare1}
\left(1+O(r)\right)\sup_{|y|=\frac{r}{1+cMr}}u(y)\le u(x)\le \left(1+O(r)\right)\inf_{|y|=\frac{r}{1-cMr}}u(y)\quad\mbox{for all }|x|=r
\ee
if $r$ is sufficiently small. Suppose $\bar y_r$ be such that $|\bar y_r|=\frac{r}{1-cMr}$ and 
\[
u(\bar y_r)=\inf_{|y|=\frac{r}{1-cMr}}u(y).
\]
Let $\bar y'_r=\frac{r}{1+cMr}\cdot \frac{\bar y_r}{|\bar y_r|}$. It follows from \eqref{eq:ln est} that
\[
\log u(\bar y_r)-\log u(\bar y'_r)\le \frac{C}{r}|\bar y_r-\bar y'_r|\le Cr.
\]
Consequently,
\be\label{eq:compare2}
\inf_{|y|=\frac{r}{1-cMr}}u(y)=u(\bar y_r)\le e^{Cr}u(\bar y'_r)\leq (1+O(r))\sup_{|y|=\frac{r}{1+cMr}}u(y).
\ee
By \eqref{eq:compare1} and \eqref{eq:compare2}, we have
\[
u(x)\le (1+O(r))u(x')\quad\mbox{for all }|x|=|x'|=r.
\]
Thus,
\[
u(x)=(1+O(r))\bar u(|x|).
\]
 Theorem \ref{thm:b} is proved. 
\end{proof}

\section{A Harnack inequality}\label{sec:harnack}

The proof of Theorem \ref{thm:c} uses again blow up analysis, which is similar to that of Theorem \ref{thm:a}. However, the blow up solutions in the proof of Theorem \ref{thm:a} come from a single given solution. But here, we have a sequence of blow up solutions which is not from any given function. To deal with this difference, the following lemma will be used.

\begin{lem}\label{lem:simple} Suppose $U\geq 0$ in $\B_2^+$ satisfies
\[
\begin{split}
\begin{aligned}
\mathrm{div}(t^{1-2\sigma} \nabla U)&=0&\quad &\mbox{in } \mathcal{B}^+_2,\\
 U(x,0)&= u(x) &\quad& \mbox{on } \pa'\mathcal{B}^+_2.
 \end{aligned}
\end{split}
\]
Then there exists a positive constant $c$ depending only on $n$ and $\sigma$ such that
\[
\inf_{\B^+_1}U\ge c\inf_{B_2}u.
\] 
\end{lem}

\begin{proof} Let $\eta$ be a smooth cut-off function supported in $B_2$ so that $\eta\equiv 1$ in $B_1$. Let $V$ be the solution of
\[
\begin{split}
\mathrm{div}(t^{1-2\sigma} \nabla V)&=0\quad \mbox{in } \mathcal{B}^+_2,\\
 V(x,0)&=\eta(x) \quad \mbox{on } \pa'\mathcal{B}^+_2,\\
 V(x,0)&=0 \quad \mbox{on } \pa''\mathcal{B}^+_2.
\end{split}
\]
Let $\tilde V=(\inf_{B_2} u)V$. By the comparison principle, we have $ U\ge \tilde V$.
It follows that 
\[
 \inf_{\mathcal{B}^+_1} U\ge \inf_{\mathcal{B}^+_1} \tilde V=c\inf_{B_2} u,
\]
where $c=\inf_{\B^+_1}V>0$.
\end{proof}

\begin{proof}[Proof of Theorem \ref{thm:c}] 
We only need to prove it for $R=1$ by making a transformation $U(X)\mapsto R^{\frac{n-2\sigma}{2}}U(RX)$. Suppose the contrary that there exists a sequence of solutions $U_j$ of \eqref{eq:nonsingu} such that
\[
u_j(x_j)\min_{\overline B_2} u_j>j\quad \mbox{as } j\to\infty,
\]
where $u_j(x_j)=\max_{\overline B_1}u_j$.

Consider
\[
v_j(x):=\left(1-|x-x_j|\right)^{\frac{n-2\sigma}{2}} u_j(x),\quad |x-x_j|\leq 1.
\]
Let $|\bar x_j-x_j|<1$ satisfy
\[
v_j(\bar x_j)=\max_{|x-x_j|\leq 1}v_j(x),
\]
and let
\[
2\mu_j:=1-|\bar x_j-x_j|.
\]
Then
\[
0<2\mu_j\leq 1\quad\mbox{and}\quad 1-|x-x_j|\ge\mu_j \quad \forall ~ |x-\bar x_j|\leq \mu_j.
\]
By the definition of $v_j$, we have
\be\label{eq:cl3-5}
(2\mu_j)^{\frac{n-2\sigma}{2}}u_j(\bar x_j)=v_j(\bar x)\ge v_j(x)\ge (\mu_j)^{\frac{n-2\sigma}{2}}u_j(x)\quad \forall ~ |x-\bar x_j|\leq \mu_j.
\ee
Thus, we have
\[
2^{\frac{n-2\sigma}{2}}u_j(\bar x_j)\ge u_j(x)\quad \forall ~ |x-\bar x_j|\leq \mu_j.
\]
We also have
\be\label{eq:cl4-5}
(2\mu_j)^{\frac{n-2\sigma}{2}}u_j(\bar x_j)=v_j(\bar x_j)\ge v_j(x_j)= u_j(x_j)\to \infty.
\ee
Now, consider
\[
W_j(y,t)=\frac{1}{u(\bar x_j)}U_j\left(\bar x_j+\frac{y}{u(\bar x_j)^{\frac{2}{n-2\sigma}}}, \frac{t}{u(\bar x_j)^{\frac{2}{n-2\sigma}}}\right ), \quad (y,t)\in \B_{\Gamma_j}^+,
\]
where
\[
\Gamma_j=u_j(\bar x_j)^\frac{2}{n-2\sigma},
\]
and let $w_j(x)=W_j(x,0)$. Then $w_j(0)=1$ and $W_j$ satisfies that 
\be \label{eq:ext1-5}
\begin{cases}
\begin{aligned}
\mathrm{div}(t^{1-2\sigma} \nabla W_j)&=0 & \quad& \mbox{in }\B_{\Gamma_j}^+,\\
\frac{\pa W_j}{\pa \nu^\sigma}& =W_j^{\frac{n+2\sigma}{n-2\sigma}} &\quad& \mbox{on }\pa' \B_{\Gamma_j}^+.
\end{aligned}
\end{cases}
\ee
Moreover, it follows from \eqref{eq:cl3-5} and \eqref{eq:cl4-5} that 
\[
w_j(y)\leq 2^{\frac{n-2\sigma}{2}} \quad\mbox{in } B_{R_j},
\] 
where $R_j:=\mu_j u(\bar x_j)^{\frac{2}{n-2\sigma}}\to \infty$ as $j\to \infty$. Then as before, after passing to a subsequence, we have, for some nonnegative function
$W\in W^{1,2}_{loc}(t^{1-2\sigma},\overline{\mathbb{R}^{n+1}})\cap C^{\al}_{loc}(\overline{\mathbb{R}^{n+1}})$
\[
\begin{cases}
W_j&\rightharpoonup W\quad\mbox{weakly in }W^{1,2}_{loc}(t^{1-2\sigma},\R^{n+1}_+),\\
W_j&\rightarrow W\quad\mbox{in }C^{\al/2}_{loc}(\overline{\R^{n+1}_+}),\\
w_j&\rightarrow w\quad\mbox{in }C^2_{loc}(\R^n).
\end{cases}
\]
Moreover, $W$ satisfies \eqref{eq:ext2}, and $w$ is as in \eqref{eq:cl5} up to some multiple, translation and scaling.

On the other hand, we are going to show that
\be\label{eq:aim1-5}
w_{\lda, x}(y)\leq w(y)\quad \forall~\lda>0, x\in \R^n, ~ |y-x|\ge\lda.
\ee
Again, by an elementary calculus lemma in \cite{LZhang}, \eqref{eq:aim1-5} implies that $w\equiv constant$, which contradicts to \eqref{eq:cl5}.

We have, with the help of Lemma \ref{lem:simple}, 
\[
\min_{\pa'' \B^+_{\Gamma_j/2}} W_j= \inf_{\B^+_{\Gamma_j/2}} W_j\ge c \inf_{B_{\Gamma_j}} W_j(\cdot,0)\ge \frac{c\min_{\overline B_2} u_j}{u_j(\bar x_j)}=\frac{c u_j(x_j)\min_{\overline B_2} u_j}{u_j(x_j)u_j(\bar x_j)}\ge c\frac{j}{u_j(\bar x_j)^2}.
\]
Thus, for any fixed $\lda_2$ and $X_0$, we have, for $j$ large,
\be\label{eq:key to move}
W_j\geq \left(\frac{\lda_2}{|\xi-X_0|}\right)^{n-2\sigma}\inf\limits_{\partial'' \B_{\lda_2}(X_0)} W_j \quad \mbox{on }\pa'' \B^+_{\Gamma_j/2}.
\ee
Once we have \eqref{eq:key to move} for $j$ large, we can show \eqref{eq:aim1-5} by the same arguments as before.
\end{proof}

\bigskip

\noindent L. Caffarelli

\noindent Department of Mathematics, University of Texas at Austin\\
1 University Station, C1200, Austin, Texas 78712, USA\\[1mm]
Email: \textsf{caffarel@math.utexas.edu}

\medskip

\noindent T. Jin

\noindent Department of Mathematics, The University of Chicago\\
5734 S. University Avenue, Chicago, IL 60637, USA\\[1mm]
Email: \textsf{tj@math.uchicago.edu}

\medskip

\noindent Y. Sire

\noindent Universit\'e Aix-Marseille and LATP\\
 9, rue F. Joliot Curie, 13453 Marseille Cedex 13, France \\[1mm]
Email: \textsf{sire@cmi.univ-mrs.fr}

\medskip

\noindent J. Xiong

\noindent Beijing International Center for Mathematical Research, Peking University\\
Beijing 100871, China\\[1mm]
Email: \textsf{jxiong@math.pku.edu.cn}

\end{document}